\documentclass[12pt,a4paper]{amsart}
\usepackage{amsaddr}
\usepackage[margin=1.2in, heightrounded]{geometry}

\usepackage[utf8]{inputenc} % allow utf-8 input
\usepackage[T1]{fontenc}    % use 8-bit T1 fonts
\usepackage{lmodern}
\usepackage{hyperref}       % hyperlinks
\usepackage{url}            % simple URL typesetting
\usepackage{booktabs}       % professional-quality tables
\usepackage{amsfonts}       % blackboard math symbols
\usepackage{nicefrac}       % compact symbols for 1/2, etc.
\usepackage{microtype}      % microtypography

\usepackage{amssymb}
\usepackage{mathtools}
\usepackage{array}
\usepackage{xcolor}
\usepackage{graphicx}
\usepackage{epstopdf}
\usepackage{algpseudocode,algorithm,algorithmicx}  
\algrenewcommand\algorithmicrequire{\textbf{Precondition:}}  
\algrenewcommand\algorithmicensure{\textbf{Postcondition:}}
\usepackage{subcaption}

\usepackage{amsthm}

\newtheorem{theorem}{Theorem}
\newtheorem{lemma}{Lemma}

\newtheorem{assumption}{Assumption}
\newtheorem*{assumption*}{Assumption}
\theoremstyle{definition}
\newtheorem*{example*}{Example}
\theoremstyle{remark}
\newtheorem*{remark*}{Remark}

\newcommand{\E}{\mathbb{E}}
\newcommand{\Pro}{\mathbb{P}}
\newcommand{\indep}{\perp \!\!\! \perp}
\newcommand{\N}{\mathcal{N}}

\newcommand{\ip}[1]{\left\langle #1 \right\rangle} %for inner product
\newcommand{\avg}[1]{ #1 _{avg}} %for inner product
\newcommand{\R}{\mathbb{R}}
\newcommand{\Z}{\mathbb{Z}}

\newcommand{\aV}{\overline{V}}

\DeclareMathOperator{\spanof}{span}
\DeclareMathOperator{\grad}{grad}
\DeclareMathOperator{\diverg}{div}
\DeclareMathOperator{\cor}{Cor}
\newcommand{\mdd}{\,\middle|\,}
\newcommand{\cT}{\mathcal{T}}

\usepackage{enumitem}
\setlist[enumerate]{leftmargin=15pt}
\setenumerate[2]{labelindent=0pt,itemindent=5pt}
\setlist[itemize]{leftmargin=15pt}

\usepackage{lineno}

\usepackage{hyperref}
\hypersetup{
    colorlinks=true,
    linkcolor=black,
    filecolor=black,      
    urlcolor=black,
    citecolor=black,}
\newcommand{\re}[1]{\textcolor{black}{#1}}
\newcommand{\ree}[1]{\textcolor{black}{#1}}
\newcommand{\nax}[1]{\textcolor{black}{#1}}
\newcommand\reout[1]{}
\newcommand\reeout[1]{}
\newcommand\naxout[1]{}

\title[linear response by nonintrusive shadowing] {
Approximating linear response by nonintrusive shadowing algorithms 
}

\author{Angxiu Ni}
\address{Beijing International Center for Mathematical Research, Peking University, Beijing 100871, P. R. China}
\email{niangxiu@gmail.com}
\urladdr{\url{https://bicmr.pku.edu.cn/~niangxiu/}}
\date{\today}

\begin{document}

\begin{abstract}
Nonintrusive shadowing algorithms efficiently compute 
$v$, the difference between shadowing trajectories,
then use $v$ to compute derivatives 
of averaged objectives of chaos with respect to parameters of the dynamical system. 
However, previous proofs of shadowing methods 
wrongly assume that shadowing trajectories are representative. 
In contrast, the linear response formula is proved rigorously, 
but is more difficult to compute.

We prove that $v$ gives only a part, called the shadowing contribution, of the linear response;
hence, the other part, the unstable contribution, is the systematic error of shadowing methods.
For systems with a small ratio of unstable dimensions,
with some further statistical assumptions,
we show that the unstable contribution is small.
We also briefly describe an algorithm for the unstable contribution,
which is simpler to derive but less efficient than the fast linear response algorithm.

Moreover, we prove the convergence of the nonintrusive shadowing algorithm,
the fastest shadowing algorithm,
to $v$ and to the shadowing contribution.
\end{abstract}

\maketitle

\section{Introduction}

In chaotic systems, while instantaneous snapshots seem random and unpredictable,
the averaged behavior is deterministic,
and can be predicted using the parameters of the system.
This means that the averaged behavior of chaos,
measured by the average of an instantaneous objective function,
\re{typically} varies smoothly to the parameters of the system, 
and the derivative is well-defined.
This derivative is fundamental to analytical and numerical tools widely used in many disciplines, 
such as gradient-based optimizations and causal inferences.
Two of the major competitors for numerical differentiation of chaos
are the linear response formula and \re{the} shadowing method.

The linear response formula gives derivatives of averaged objective in hyperbolic systems, 
which is typically used as a model for general chaotic systems
\cite{Ruelle_diff_maps,Ruelle_diff_maps_erratum,Ruelle_diff_flow,Jiang2012}.
When the dynamical system has certain kind of hyperbolicity, say uniform hyperbolicity, 
the system is guaranteed to have structural stability under perturbations,
and we can prove that the linear response formula gives the correct derivative.
It should be noted that linear response fails for certain systems \cite{Baladi2007}.
In particular, the derivative may not even exist
for systems going through radical structural changes, such as bifurcations.

In terms of computations, the original linear response formula 
can be directly implemented in an ensemble approach or an operator-based approach
\cite{Lea2000,eyink2004ruelle,lucarini_linear_response_climate,
lucarini_linear_response_climate2,Bahsoun2018,Crimmins2020,Galatolo2014,Galatolo2014a}.
These algorithms converge slowly,
due to averaging out an exponentially growing integrand \cite{Chandramoorthy_ensemble_adjoint},
or inefficiency in approximating singular measures via isotropic finite elements in higher dimensions \cite{TrsfOprt}. 
On the other hand, via integration-by-parts on the unstable manifold,
we can get an alternative linear response formula with \ree{a} much smaller integrand,
which involves \re{the} divergence on unstable manifolds \cite{Gallavotti1996,Ruelle_diff_maps_erratum}.
The unstable divergence is very difficult to compute, 
since the directional derivatives are typically infinite.
Various approximations were introduced for computing the unstable divergence,
such as the blended response algorithm \cite{abramov2007blended,Abramov2008}.
Blended response is more efficient than 
the original linear response formula, yet still quite expensive,
and so far has been demonstrated only on systems with less than 100 dimensions.
The error analysis for blended response was previously missing,
and our error bound for shadowing will also provide a bound for the blended response.

Shadowing methods, starting from the theoretical advancement made by Anosov, Bowen, and Pilyugin
\cite{anosov_shadow,Bowen_shadowing,Pilyugin_shadow_linear_formula},
were used for numerical differentiation of chaos
\cite{wang2014convergence,Shimizu2018,Gunther2017,Blonigan_2017_adjoint_NILSS,Shawki2019,Lasagna2019,Blonigan_MSS}.
Shadowing methods first compute the shadowing direction, $v$,
which is the difference between shadowing trajectories;
then shadowing methods use $v$ to compute the derivative.
The shortcoming of the shadowing method is that it makes the strong assumption that
shadowing trajectories are representative of the long-time-averaged statistic of the perturbed system.
This is not true in general, and shadowing methods can fail for simple systems 
such as the 1-dimensional expanding circle \cite{BloniganPhdThesis}.
There were some very insightful discussions on the systematic error 
of shadowing methods by Blonigan \cite{Blonigan2014a,BloniganPhdThesis},
whose main difference from our work is 
the missing of a correct formula to which we can compare.

Hence, it is of interest to rebuild the theoretical foundation of shadowing methods
by comparing it with the linear response formula.
As we shall see in this paper, the result of shadowing methods gives a part of the correct derivative,
which we call the shadowing contribution of the linear response.
Moreover, we show that the shadowing shadowing contribution
is a good approximation of the entire linear response for some interesting cases,
such as high-dimensional systems with low-dimensional attractors.
This partially explains the success of shadowing in contexts such as fluid mechanics.

The computational efficiency and ease of implementation of shadowing methods
were significantly improved by a `nonintrusive' formulation.
`Nonintrusive' means that it uses only $m$-many solutions of the most basic equation, 
where $m$ is the unstable dimension, and the most basic equation in this case is the tangent equation. 
This gives the nonintrusive shadowing algorithm \cite{Ni_NILSS_JCP,Ni_fdNILSS}.
Continuous-time and adjoint versions of nonintrusive shadowing algorithms
have also been developed \cite{Ni_adjoint_shadowing,Ni_nilsas,Blonigan_2017_adjoint_NILSS}.
It is so far the only algorithm that has been demonstrated in 
very high dimensional problems,
such as a computational fluid problem with $4\times10^6$ dimensions \cite{Ni_CLV_cylinder}.
The efficiency improvement is because the nonintrusive formulation 
reduces the computation to the unstable subspace.

It is hence of interest to ask how much error is caused by the nonintrusive reduction.
The \ree{latter} part of this paper shows that 
nonintrusive shadowing is accurate for computing $v$.
More specifically, it has the same convergence-to-zero speed 
\re{as} previous shadowing methods.
Together with the first part of the paper, 
we give an error analysis of approximating \ree{the} linear response by nonintrusive shadowing.

Moreover, this paper is the first step towards the fast linear response algorithm.
This paper shows that the linear response can be decomposed into 
the shadowing contribution and the unstable contribution,
and that the shadowing contribution can be computed by nonintrusive shadowing.
Then, in another paper, we give a fast algorithm for computing the unstable contribution,
via a `fast' characterization by second-order tangent equations,
whose second derivative is taken in a modified shadowing direction \cite{flr}.
The fast linear response algorithm seems to be the fastest accurate algorithm
for the linear response of chaos.
Its derivation is quite complicated,
but the procedure list is not too much more complicated:
the main extra procedure, compared to nonintrusive shadowing,
is solving $m$-many second-order tangent equations,
which are first-order tangent equations with a second-order inhomogeneous term.
The fast linear response algorithm uses nonintrusive shadowing twice,
once for computing the shadowing contribution, 
once for the modified shadowing direction in the unstable contribution.
It can be somewhat surprising that nonintrusive shadowing is also important 
for efficiently computing the unstable contribution.

\re{
With the development of the fast linear response algorithm,
it is of even more interest to analyze the error of nonintrusive shadowing,
which is the main result of this paper.
It is also of interest to compute the unstable contribution
with simpler derivations, which typically means easier generalizations,
such as to continuous-time and adjoint versions.
However, here the flip side of a simple derivation is low efficiency or large error.
Such an algorithm is briefly described in this paper.
}

% structure of paper
This paper is organized as follows.
First, we review the shadowing method and the linear response formula for discrete systems.
Then we prove that the shadowing method gives only a part of the linear response,
which we call the shadowing contribution.
Moreover, with two statistical assumptions,
we show that remaining part, the unstable contribution, of the linear response,
is positively related to the ratio of unstable dimension to the dimension of the system.
The two assumptions are (1) fast decay of correlations, and 
(2) both the gradient of the objective function and the perturbation of the dynamical system 
are not particularly aligned with the unstable subspace.
We also explain how to compute part of the unstable contribution by a simple derivation.
Finally, we prove the convergence of the nonintrusive shadowing algorithm to 
the shadowing direction, $v$, and to the shadowing contribution.

\section{Preparations}

\subsection{Hyperbolic dynamical systems}
\hfill\vspace{0.1in}

% primal system
Consider an autonomous system with the governing equation:
\begin{equation} \label{e:primal_system_discrete}
  u_{k+1} = f(u_k, \gamma), \quad k\ge 0\,.
\end{equation}
Here $f$ is a $C^\infty$ diffeomorphism in $u$, state of the dynamical system,
where $u\in \R^M$; $\gamma\in\R$ is the parameter.
We consider only the case where the phase space is Euclidean, 
for the convenience of posing a statistical model later on,
which is used to quantitatively bound the error.
We may as well extend our results to chaos on Riemannian manifolds.
Also notice that we typically use $n, k, m$ to label steps,
and $i, j$ to label directions in the phase space.

The objective, $\avg{\Phi}$, is a long-time-averaged quantity
which converges to the same value for almost all initial conditions,
\begin{equation} \label{e:average_J_discrete}
  \avg{\Phi}= \lim\limits_{K\rightarrow\infty} \frac{1}{K} \sum_{k=0}^{K-1} \Phi(u_k) ,\, \quad a.e.
\end{equation}
Here $\Phi$ is a smooth function representing the instantaneous objective.
The goal is to perform sensitivity analysis, that is,
to compute the derivative $\delta \avg{\Phi}$, where
\[ \begin{split}
  \delta (\cdot) :=  \partial (\cdot) /\partial \gamma
\end{split} \]
may as well be thought of as small perturbations caused by changing $\gamma$.
We assume that $\Phi$ is fixed as $\gamma$ varies;
if not so, we only need to add the average of $\delta \Phi$ to the linear response.

% tangent equations
To compute the derivative of the averaged objective, 
we first investigate how perturbing the parameter would affect individual trajectories.
Differentiate equation~\eqref{e:primal_system_discrete} with respect to $\gamma$, 
define $v_k:=\delta u_k$, 
it satisfies the inhomogeneous tangent equation:
\begin{equation}\label{e:inhomo_tangent_diffeo}
  v_{k+1} = f_{*} v_k + X_{k+1}\,.
\end{equation}
where $X:=\delta f \circ f^{-1}$ is a smooth vector field,
and $X_{k+1}=\delta f(u_k)$ is a column vector.
Here $f_*$ is the pushforward operator on vectors.
In this paper, $f_*$ is the pushforward operator,
which applies on vectors or measures.
In $\R^M$, applying $f_*$ on vectors can be represented by multiplying by the Jacobian matrix,
$\partial f/\partial u$;
on the other hand, applying $f_*$ on measures is represented by multiplying the density function
with $\det(\partial f/\partial u)^{-1}$.
The initial condition $v_0$ is yet to be determined, 
since there is some freedom to choose $u_0$ without affecting the objective.

A homogeneous tangent solution, $\{w_k\}_{k\in \Z}$, 
where $w_k$ is a vector at $u_k$,
is the solution of the homogeneous tangent equation,
\begin{equation}\label{e:homo_tangent_diffeo}
  w_{k+1} = f_{*} w_k \,.
\end{equation}
This equation governs \ree{a} perturbation on a trajectory
caused by perturbing the initial condition;
unlike the inhomogeneous version, here $\gamma$ is fixed.

% hyperbolicity
This paper assumes uniform hyperbolicity, that is, for every $u$ on a compact invariant set $\cT$,
there is a splitting of the tangent space at $u$, $\R^M(u)=V^+(u) \bigoplus V^-(u)$, 
where $V^+$ is the unstable subspace of dimension $m$, and $V^-$ the stable subspace.
Moreover, there is a constant $C_1\ge 1$ and $\lambda\in (0,1)$ such that,
\begin{equation} \begin{split}\label{e:tangent decay}
  \| f_*^k w \| \le C_1 \lambda^{-k}\| w \|,\; \textnormal{ for }\;
  k\le 0, w \in V^+, \\
  \| f_*^k w \| \le C_1 \lambda^{k}\| w \|,\; \textnormal{ for }\; 
  k\ge 0, w \in V^-.
\end{split} \end{equation}
We further assume the hyperbolic set $\cT$ is an attractor,
that is, there is an open neighborhood $\cT'$ of $\cT$
such that $\cap_{n\ge0} f^n(\cT')=\cT$.

% SRB measure
Uniform hyperbolic systems have the SRB measure,
\re{which is the fractal limiting stationary measure of chaotic systems,
named after Sinai, Ruelle, and Bowen}
\cite{Sinai1972,srbmap,srbflow}.
It has several characterizations, 
and for this paper, 
we define it as the weak limit of evolving Lebesgue measures \cite{young2002srb}.
That is,
\[ \begin{split}
  \rho = \lim_{n\rightarrow\infty} f_*^n \rho_0,
\end{split} \]
where $\rho_0$ is the Lebesgue measure, and 
\re{here $f_*$ is the pushforward operator on measures.}
Hence, for almost all $u_0$ in a neighborhood of the attractor,
the empirical distribution of the trajectory starting from $u_0$
weakly converges to the SRB measure, 
and $\avg{\Phi}$ is in fact defined as
\begin{equation} \begin{split} \label{e:ae}
  \avg{\Phi} := \rho(\Phi) .
\end{split} \end{equation}
Hence, our goal is to differentiate the SRB measure,
that is, to compute $\delta \rho$.

% covariance
Finally, we define equivariant sequences.
A sequence, say $\{v_k\}_{k\in \Z}$, 
depends on the underlying trajectory, in particular its initial condition, $u_0$.
We typically do not write out $u_0$ explicitly as a variable of $v_k$,
but when computing integrations such as $\rho(v_k)$, 
we let $u_0$ distribute according to $\rho$.
In this paper, a sequence is said to be equivariant if its 
evolution commutes with the evolution of the initial condition,
that is,
\[ \begin{split}
  v_k(u_0) = v_0(u_k).
\end{split} \]
For equivariant sequences, 
due to the invariance of SRB measures,
\begin{equation} \begin{split} \label{e:cov}
  \rho(v_k):=\int v_k(u_0) \rho (du_0) 
  = \int v_0(u_k) \rho (du_0)
  = \int v_0(u_0) \rho (du_0)
  =: \rho(v_0).
\end{split} \end{equation}
If given a function, say $g$, then $g_k(u_0):=g(u_k)$ is equivariant by definition.
In this paper, some sequences are equivariant, such as the shadowing direction $v$,
and later $v^A$;
however, some are not equivariant, such as $e^P, e^N, e^{PN}$, and $v^p$.
It is important to apply equation~\eqref{e:cov} only on equivariant sequences.

\subsection{Shadowing methods}
\label{s:shadowing methods}
\hfill\vspace{0.1in}

% shadowing direction
Uniform hyperbolic systems have the shadowing property.
Roughly speaking, after perturbing the parameter by $\Delta \gamma$, 
we can shift each state by a small amount, $v_k\Delta \gamma$, 
to obtain a new trajectory of the perturbed system,
which is called the shadowing trajectory \cite{Bowen_shadowing,anosov_shadow}.
Hence, although most inhomogeneous solutions grow exponentially fast,
there is a special inhomogeneous tangent solution, the shadowing direction,
whose norm remains bounded.

We first write out an explicit formula of the shadowing direction.
At each step, split $X$ into stable and unstable components, 
and propagate the stable component into the future, the unstable component into the past. 
More specifically, 
\begin{equation} \label{e:shadowing_diffeo}
  v_{k} = \sum_{n\ge 0}f_*^nX_{k-n}^- - \sum_{n\le -1} f_*^n X_{k-n}^+ \,,
\end{equation}
Here $X^- := P^- X$, $X^+ := P^+ X$,
where $P^-$ and $P^+$ are oblique projection operators
onto the stable and unstable subspace.
Due to the exponential decay of stable and unstable components,
both summations converge and $v$ is bounded.

% assumption of ergodicity
To use the shadowing property for computing derivatives, 
shadowing methods make an extra assumption that
shadowing trajectories are representative of the perturbed system.
That is, for the perturbed system,
$\avg{\Phi}:=\rho(\Phi)$ can be computed from the shadowing trajectory.
This is a very strong assumption,
since it essentially says that the new system is so similar to the old system
that the old behavior is shadowed;
it is equivalent to the existence of a smooth conjugation map between the two systems.
A conjugation map does exist, but it is not smooth enough to preserve representative behaviors.
Hence, the extra assumption is typically false;
it causes an error, which will be examined in section~\ref{s:relation}.

% how to use shadowing direction to compute sensitivity, why assumption?
For now, we assume that shadowing trajectories are representative of the long-time behavior;
hence, we can take their difference to compute the change in the averaged objective.
Due to the boundedness of the shadowing directions, 
the limit of summation and the limit in the derivative can interchange place, so
\begin{equation} \label{e:dJds_diffeo}
  \delta \avg{\Phi} 
  \approx \delta \Big( \lim\limits_{K\rightarrow\infty} 
    \frac{1}{K} \sum_{k=0}^{K-1} \Phi(u_k) \Big)
  = \lim_{K\rightarrow \infty} \frac 1K \sum_{k=0}^{K-1} 
  \Phi_{uk} v_{k} 
  \stackrel{a.e.}{=} \rho(\Phi_uv)
  =: \delta^{sd} \avg{\Phi}
  \,,
\end{equation}
where $\Phi_{uk}:=\partial \Phi/\partial u(u_k)$ is a row vector;
\nax{it is, in fact, a differential form.}
Here the approximation sign reflects the error introduced by the extra assumption,
and $\stackrel{a.e.}{=}$ means to hold almost everywhere on the basin of the attractor.
\nax{
The target of all shadowing methods is to compute $\delta^{sd}\Phi_{avg}$ via first computing $v$.
We say `target' here,
because the convergence of nonintrusive shadowing is not yet justified;
we will prove this convergence in section~\ref{s:NIS}.
We shall also see that the target of shadowing methods equals the shadowing contribution of the linear response.
}

To efficiently compute shadowing directions,
we first notice that the seemingly complicated formula in equation~\eqref{e:shadowing_diffeo} 
can be equivalently characterized by:

\begin{lemma}
For any trajectory on the attractor,
the shadowing direction is the only inhomogeneous tangent solution that is bounded for all time.
\end{lemma}

The nonintrusive shadowing algorithm
%also know as the nonintrusive least-squares shadowing (NILSS) algorithm,
recovers above characterization by a constrained minimization.
The boundedness property is mimicked by minimizing the $l^2$ norm of $v$.
The fact that $v$ is an inhomogeneous tangent solution 
is recovered by the representation as the sum of a particular inhomogeneous
and several homogeneous tangent solutions.
More specifically, the nonintrusive shadowing algorithm solves
\begin{equation} \begin{split} \label{e:nilss}
  \min_{\{a_j\}_{j=1}^m\subset \R} \sum_{k=0}^{K-1} |v_k|^2 ,
  \quad\textnormal{ s.t. }\quad
  v = v' + \sum_{j=1}^{m} w_ja_j.
\end{split} \end{equation}
where $|\cdot|$ is the vector 2-norm, $K$ is the trajectory length;
$v'$ is an inhomogeneous tangent solution of any initial condition,
for example zero initial conditions;
$\{w_j\}_{j=1}^{m}$ are $m$ homogeneous tangent solutions with random initial conditions
\cite{Ni_NILSS_JCP,Ni_fdNILSS}.

Nonintrusive shadowing does not search the entire space of inhomogeneous solutions, 
which is $M$-dimensional.
Rather, the feasible set is reduced to a subspace of dimension $m$.
Such a reduced feasible set is still enough for us to find a bounded solution:
since $v'$ is solved by pushing-forward in time,
the only cause for its exponential growth is the unstable component.
This unstable component can be removed by a linear combination of $w_j$'s,
which also approximates the unstable subspace after pushing-forward for some time.
Section~\ref{s:NIS} quantitatively shows that this reduction of the feasible set causes no additional error.

Nonintrusive shadowing is the first numerical differentiation algorithm of chaos
whose computation is constrained to the unstable subspace:
this is achieved by the `nonintrusive' parameterization we used in equation~\eqref{e:nilss}.
`Nonintrusive' means that we use only $m$ many solutions of the most basic governing equations,
which is the tangent equation for this case,
but no other information such as the Jacobian matrices.
%This parameterization allows us to handle each tangent solutions as a whole,
%and use them to approximate the unstable subspace, 
%and to remove the unstable components in $v'$.
For cases with $m\ll M$, such as computational fluid problems,
nonintrusive shadowing is thousands of times faster than previous algorithms.
\ree{
For the examples we have so far,
the cost of nonintrusive shadowing is similar to the numerical simulation of the system
\cite{Ni_NILSS_JCP,Ni_CLV_cylinder}.
When the unstable dimension gets larger,
the cost of nonintrusive shadowing, per trajectory length,
can be larger than the simulation.
However, it seems that nonintrusive shadowing requires a shorter trajectory
than the averaged objective given by the simulation;
moreover, 
the computation of $m$ many tangent solutions can be greatly accelerated by a vectorized code.
Overall, for very unstable problems,
the cost of nonintrusive shadowing might still be comparable to simulations,
but more experiments are needed to verify or disprove this claim.
}

\subsection{Linear response formula} 
\hfill\vspace{0.1in}

% How to derive, and the first formula
In shadowing methods, the exponential growth of inhomogeneous tangent solutions
is tempered by granting some freedom in its initial condition, then minimizing its norm.
Another way to temper this exponential growth is to average over SRB measures.
By some formal interchange of limits, we can show that
\begin{equation}\label{e:ruelle1}
  \delta \avg{\Phi}
  = \sum_{n=0}^\infty \rho \ip{\grad (\Phi\circ f^n ), X} \,.
\end{equation}
Here $\ip{\cdot , \cdot }$ is the inner product in $\R^M$,
$\rho$ is the SRB measure, and $\delta:=\delta/\delta \gamma$.
\nax{
In $\R^M$, the gradient, $\grad \Phi = \Phi_u^T$, is a column vector, 
where $(\cdot)^T$ the matrix transposition.
}
By a different derivation, 
Ruelle and Dolgopyat proved that this formula indeed gives 
the correct derivative for uniformly hyperbolic and partially hyperbolic systems 
\cite{Ruelle_diff_maps,Dolgopyat2004}.

It should be noted that the linear response fails for certain systems,
for example, the tent map,
which is essentially nonuniformly hyperbolic \cite{Baladi2007}.
It certainly remains to be investigated when and how often the linear response fails,
but we should also notice that some arguments are invalid.
In particular, it is typical for shadowing trajectories to
differ from the long-time-averaged statistic of the perturbed system,
and the linear response formula already accounts for that.
As a related issue, due to numerical errors,
solutions of numerical simulations are typically shadowed by 
non-representative solutions of the true physical system,
but that numerical solution is still typically representative of the numerical system,
whose statistics are close to the true system.
This is perhaps one of the most basic assumptions of numerical simulations,
and it can hold true even when the linear response fails.

Numerically, the linear response formula can be directly implemented by an ensemble approach
\cite{Lea2000,eyink2004ruelle,lucarini_linear_response_climate,
lucarini_linear_response_climate2,Bahsoun2018}.
However, the integrand grows exponentially to $n$,
and the number of samples needed to evaluate the integration of $\rho$,
to a certain precision, is very large,
incurring large computational cost \cite{Chandramoorthy_ensemble_adjoint}.

% the second formula
To temper the large integrand in equation~\eqref{e:ruelle1},
we integrate by parts on the unstable manifold \cite{Gallavotti1996}, 
so that $\rho \ip{\grad (\Phi\circ f^n),X^+} = -\rho( (\Phi\circ f^n) \diverg_\sigma^+ X^+) $,
and
\begin{equation} \label{e:ruelle2}
  \delta \avg{\Phi}
  = \sum_{n=0}^\infty \rho\left[  \ip{\grad (\Phi\circ f^n ),X^-} 
  - (\Phi\circ f^n ) \diverg_\sigma^+  X^+ \right].
\end{equation}
Here $\diverg_\sigma^+$ is the divergence on the unstable manifold under the conditional SRB measure.
By definition, $\diverg_\sigma^+  X^+$ is a distribution,
but Ruelle showed that it is Holder continuous on a uniform hyperbolic attractor
\cite{Ruelle_diff_maps_erratum}.
For a more detailed discussion of this term in the context of computations, see \cite{flr}.

% good and bad of the second formula
Equation~\eqref{e:ruelle2} circumvents the issue of exploding gradients,
since the first term involves propagating only the stable components into the future,
while the second term is subject to the exponential decay of correlation.
That is, because both $\Phi$ and $\diverg_\sigma^+  X^+$ are Holder continuous,
there is $C_2'>0$ and $\kappa_2\in(0,1)$, such that
\begin{equation} \label{e:decay1}
  \cor_{\Phi,\diverg_\sigma^+  X^+}(n) :=
  \left|\rho((\Phi\circ f^n ) \diverg_\sigma^+ X^+)
  -\rho(\Phi )\rho(\diverg_\sigma^+ X^+) \right|
  \le C_2' \kappa_2^n.
\end{equation}
Since $\rho (\diverg_\sigma^+  X^+) = 0$,
we have $\cor_{\Phi,\diverg_\sigma^+  X^+}(n) 
= \left|\rho((\Phi\circ f^n ) \diverg_\sigma^+ X^+) \right|$.
It is very convoluted to express $C_2'$ and $\kappa_2$ by properties of the dynamical systems.
Even if we could theoretically derive such formulas,
they would be difficult to compute for engineering applications.
To obtain a quantitative bound, we make a \nax{statistical} assumption,
assumption~\ref{a2} in section~\ref{s:error}, about \ree{the} decay of correlation.

For our purpose, we use a slightly different decomposition of the linear response,
\begin{equation} \begin{split} \label{e:ds123}
  &\delta \avg{\Phi} = \delta^{(1)} \avg{\Phi}
    + \delta^{(2)} \avg{\Phi}+ \delta^{(3)} \avg{\Phi} 
    \,, \quad \textnormal{where}\\
  &\delta^{(1)} \avg{\Phi}
    := \sum_{n\ge 0} \rho \ip{\grad (\Phi\circ f^n ),X^-}
    - \sum_{n\le -1} \rho \ip{\grad (\Phi\circ f^{n} ),X^+},\\
  &\delta^{(2)} \avg{\Phi}
    := \sum_{n< N} \rho \ip{\grad (\Phi\circ f^{n} ),X^+} 
    ,\,
  \delta^{(3)} \avg{\Phi}
    := -\sum_{n\ge N} \rho \left((\Phi\circ f^n ) \diverg_\sigma^+  X^+\right).
\end{split} \end{equation}
Here $N$ is a small positive integer.
We call $\delta^{(2)}\Phi_{avg} +\delta^{(3)}\Phi_{avg}$
the unstable contribution of the linear response,
because they only involve the unstable part of $X$.
For reasons to be explained later, we also denote the unstable contribution by $\delta\mu(\Phi)$.
%\nax{
%We dissect the unstable contribution into two parts,
%because $\delta^{(3)}\Phi_{avg}$ can only be treated by the decay of correlations,
%whereas $\delta^{(2)}\Phi_{avg}$ can as well be treated by the decay of unstable vectors.}
We call \ree{$\delta^{(1)}\Phi_{avg}$} the shadowing contribution of the linear response,
because, as we shall see by theorem~\ref{t:equal},
$\delta^{(1)}\Phi_{avg} = \delta^{sd}\Phi_{avg}$.

\section{Approximating linear response by shadowing}
\label{s:relation}

In this section, 
we examine the difference between the linear response formula and 
the target of shadowing methods, $\delta^{sd}\Phi_{avg}$.
Notice that the nonintrusive formulation does not appear in this section,
and our discussion applies to all shadowing methods.
\re{Compared} 
to previous proofs of shadowing methods \cite{Chater_convergence_LSS,wang2014convergence}, 
which make the extra assumption that shadowing trajectories are representative,
here we replace that assumption by a bound 
of the remaining part of the linear response formula,
which is the unstable contribution.

\subsection{Shadowing methods' target is the shadowing contribution}
\hfill\vspace{0.1in}

% rough recap of the proof
To reveal the connection between shadowing and linear response,
we further explain how the linear response formula was proved for uniform hyperbolic systems.
When changing $\gamma$ to a new parameter, $\tilde \gamma$, 
$f$ is changed to $\tilde f:=f(\cdot,\tilde \gamma)$, 
and the SRB measure is changed to $\tilde \rho$, whose support, or the attractor, also moves.
Ruelle showed that there is a Holder diffeomorphism, $j$, 
between the two attractors, so that $\tilde f \circ j = j \circ f$.
Let $\mu(\cdot):=\tilde \rho (j(\cdot))$,
then $\mu$ has the same support as $\rho$,
and $\tilde \Phi_{avg}: =\tilde \rho (\Phi)= \mu(\Phi\circ j)$.
Differentiate with respect to $\gamma$, 
apply the product rule, we have
\[ \begin{split}
  \delta \avg{\Phi}=\rho(\delta(\Phi\circ j))+\delta \mu(\Phi)
    = \rho(\Phi_u \delta j)+\delta \mu(\Phi).
\end{split} \]
Here $\rho(\delta(\Phi\circ j))$ accounts for the change of location of the attractor;
$\delta \mu(\Phi)$ accounts for the difference between $\mu$ and $\rho$,
which are both stationary measures, but only $\rho$ is SRB.
Ruelle showed that these two terms have expressions given in equation~\eqref{e:ds123}:
\[ \begin{split}
  \rho(\delta(\Phi\circ j)) = \delta^{(1)} \avg{\Phi}  
  ,\quad
  \delta \mu(\Phi) = \delta^{(2)}\avg{\Phi} + \delta^{(3)}\avg{\Phi} .
\end{split} \]

% I1 is shadowing methods: intuition
The term $\rho(\delta(\Phi\circ j))$ is the derivative while assuming $\mu$ is fixed,
that is, assuming that the SRB measure is preserved by the conjugation map $j$.
This assumption is very similar to the assumption we made for shadowing methods, 
hinting the equivalence $\delta^{sd} \avg{\Phi}= \delta^{(1)} \avg{\Phi}$.
In fact, using the Taylor expansion of $\delta j$, Ruelle showed that $\delta j =v$,
which immediately yields this equivalence;
however, this equivalence admits a much more elementary proof which does not involve $j$.

% equivalence lemma
\begin{theorem}\label{t:equal}
  The shadowing contribution of the linear response 
  is exactly the target of shadowing methods.
  That is,
  \[ \begin{split}
  \delta^{(1)} \avg{\Phi}  = \delta^{sd} \avg{\Phi} .
  \end{split} \]
Here $\delta^{(1)} \avg{\Phi}$ is defined in equation~\eqref{e:ds123},
and $\delta^{sd} \avg{\Phi}$ is defined in equation~\eqref{e:dJds_diffeo}.
\end{theorem}

\begin{proof}
  Apply the invariance of the SRB measure, we have
\[ \begin{split}
  &\delta^{(1)} \avg{\Phi}
  = \sum_{n\ge0} \rho \left[ \ip{\grad (\Phi\circ f^n),X^-}\circ f^{-n} \right]
  - \sum_{n\le -1} \rho \left[ \ip{\grad (\Phi\circ f^{n} ),X^+}\circ f^{-n} \right].
\end{split}\]
  By the exponential decay, the above formula converges absolutely,
  hence we can use Fubini's theorem to interchange summation and integration, and
\[ \begin{split}
  \delta^{(1)} \avg{\Phi}
  = \rho \left[ \sum_{n\ge0}
  \ip{\grad (\Phi\circ f^n),X^-}\circ f^{-n} 
  - \sum_{n\le -1} \ip{\grad (\Phi\circ f^{n} ),X^+}\circ f^{-n} \right]\\
\end{split}\]
%Since SRB measure can almost surely be evaluated by long-time averages, 
%\[ \begin{split}
 %\delta^{(1)} \avg{\Phi}
  %= \lim_{K\rightarrow \infty} \frac 1K \sum_{k=0}^{K-1} \left[
  %\sum_{n\ge0}
  %\ip{\grad (\Phi\circ f^n),X^-}(u_{k-n})
  %- \sum_{n\le -1} \ip{\grad (\Phi\circ f^{n} ),X^+}(u_{k-n}) \right] \\
%\end{split}\]
By definition of pushfoward operators,
\[ \begin{split}
  \ip{\grad (\Phi\circ f^n),X^\pm}(u_{-n}) = \Phi_{u}f_*^nX^\pm_{-n}.
\end{split} \]
\[ \begin{split}
 \delta^{(1)} \avg{\Phi}
  &= \rho \left[
    \sum_{n\ge0} \Phi_{u} f_*^n X^-_{-n}
  - \sum_{n\le -1} \Phi_{u}f_*^n X^+_{-n} 
  \right] 
  = \rho ( \Phi_u v )
  = \delta^{sd} \avg{\Phi},
  \end{split} \]
where the shadowing direction, $v$, is defined in equation~\eqref{e:shadowing_diffeo}.
\end{proof}

% a condition for shadowing methods being exact: f is C1
The shadowing method is off from the correct linear response by a systematic error, 
$\delta \mu(\Phi)$, which is the unstable contribution.
A sufficient condition for this term to be zero is that 
$j$ can be extended to a $C^1$ diffeomorphism over the entire phase space.
When that happens, absolute continuity to the Lebesgue measure is preserved, 
and $\mu$ is also the limit of evolving the Lebesgue measure.
Since the SRB measure is the unique limit of evolving the Lebesgue measure,
$\mu$ must always be the SRB measure on the original attractor, 
which yields $\delta \mu \equiv 0$.
Such a $C^1$-extendable $j$ exists, for example,
when the perturbed dynamical system is obtained 
by distorting a neighborhood of the attractor via a $C^1$ map.
However, it should still be rare for 
$j$ to be $C^1$-extendable;
in fact, under uniform hyperbolicity,
$j$ is typically only a Holder homeomorphism on the attractor,
and the unstable contribution is not zero.
So instead of hoping the systematic error of shadowing to disappear,
we shall give an estimation of the unstable contribution,
and examine when it can be small.

\subsection{Statistical assumptions for estimating the unstable contribution} 
\label{s:assump}
\hfill\vspace{0.1in}

We state some assumptions to be used in the next subsection, 
where we bound the unstable contribution.
By equation~\eqref{e:ds123}, 
the unstable contribution is related to the magnitude of $\Phi_u X^+ = (\Phi_u P^+)(P^+ X)$.
Intuitively, if $X$ and $\Phi_u$ have no 
particular reason to be aligned with unstable subspaces,
projection to a low dimensional unstable subspace significantly reduces the vector norms.
Furthermore, if the decay of correlation is fast, 
we can estimate the entire unstable contribution by the leading term.
Hence, the unstable contribution should be positively related to the unstable ratio, $m/M$.

As we can see, this argument is based on two phenomena, 
which shall be stated quantitatively by two statistical assumptions in this subsection.
It should be noted that these assumptions are for quantification of the errors;
the qualitative behaviors, such as a small systematic error of shadowing,
and the convergence of nonintrusive shadowing to the shadowing contribution,
can hold true beyond these specific assumptions.
These assumptions are just one way to quantify these phenomena,
which may as well be quantified by other statements.
To conclude, our quantitative bounds are statistical results;
in particular, they base on the statistical assumptions.
More experiments are needed for verifying our assumptions,
although current available examples do suggest the assumptions,
or at least the phenomena they intend to describe, hold in some way,
and nonintrusive shadowing works well for systems with a small unstable ratio.

For fixed $X$ and $\Phi$, 
it is difficult to give a quantitative prior error bound for shadowing methods,
because computing $X^+$ is 
more expensive than nonintrusive shadowing,
at which point \ree{a priori} estimation would \ree{not be beneficial}.
To give an estimation of the shadowing error beforehand,
we view $\Phi$ and $X$ as random functions.
Then we can bound the expectation of the shadowing error under
the particular statistical model we choose for $\Phi$ and $X$.
Also, we let $U$ be a random variable distributed according to the SRB measure,
whose total measure is normalized to 1.
Choosing the random functions $\Phi$ and $X$ does not affect the dynamical system and its SRB measure; hence $\Phi$ and $X$ are independent of $U$.

\begin{assumption}\label{a1}
  For any $u$, $X(u)$ and $\Phi_u(u)$ follow multivariate normal distributions $\N(0, I_M)$.
  Moreover, for any sequence $\{u_n\}_{n\in\Z}$, 
  the sequence $\{X(u_n)\}_{n\in\Z}$ is independent of $\{\Phi_u(u_n)\}_{n\in\Z}$.
  Since $\Phi$ and $X$ are independent of $U$, we can write in conditional probability,
  \[ \begin{split}
    (X(U) \,|\, U=u) \sim \N(0, I_M),  \quad
    (\Phi_u(U) \,|\, U=u) \sim \N(0, I_M), \quad
    \forall u. \\
    \{X(U_n)\}_{n\in\Z} \indep
    \{\Phi_u(U_n)\}_{n\in\Z} 
    \,|\, \{U_n=u_n\}_{n\in\Z} ,\quad
    \forall  \{u_n\}_{n\in\Z} .
  \end{split} \]
\end{assumption}

\begin{remark*}
(1) An example satisfying this assumption is that both $X$ and $\Phi_u$ are constant 
vector fields on $\R^M$, whose values are drawn from two independent Gaussian.
(2) For our purpose, it suffices to assume only for the case where $\{u_n\}_{n\in \Z}$
is a trajectory.
\nax{
(3) Roughly speaking, lemma~\ref{t:JuX+} only needs the independence between $X$ and $\Phi_u$
at any one $u$;
theorem~\ref{t:approximation} needs the independence at two $u$'s;
theorem~\ref{t:convergeNI} and~\ref{t:sample error} need the independence on a full trajectory.
}
\end{remark*}

This assumption is for quantifying the phenomenon that $X$ and $\Phi_u$ are not particularly aligned with the unstable subspace.
This phenomenon is plausible, since typically $X$ and $\Phi_u$ are determined without any 
prior knowledge of the unstable direction.
For example, in a chaotic flow over a cylinder \cite{Ni_CLV_cylinder},
$X$ is perturbation on the inlet condition, $\Phi$ is the drag/lift on the cylinder,
hence, $X$ and $\Phi_u$ are only non-zero at the inlet and surface of the cylinder.
On the other hand, the unstable modes are active mainly in the wake of the cylinder.
The different active areas indicate that $X$ and $\Phi_u$ 
are at least not aligned with the unstable direction, 
if not exactly $\Phi_u P^+X=0$.

Then we make a statistical assumption about \ree{the} decay of correlation.
Roughly speaking, in equation~\eqref{e:decay1},
we assume the exponential decay of correlation starts from the zeroth term.

\begin{assumption}\label{a2}
For the entire distribution of $\Phi$ and $X$,
there are uniform constants $C_2\ge 1,  0<\kappa_2<1$, such that
\[ \begin{split}
  \cor_{\Phi,\diverg_\sigma^+  X^+}(n) = 
  \left|\rho((\Phi\circ f^n ) \diverg_\sigma^+ X^+) \right|
  \le C_2 \kappa_2^n \rho(|\Phi_u X^+|).
\end{split} \]
\end{assumption}

\begin{remark*}
(1)
Here $\rho(|\Phi_u X^+|)$ is a bound for the zeroth term, since
\[ \begin{split}
  \cor_{\Phi,\diverg_\sigma^+  X^+}(0) =
  \left|\rho(\Phi \diverg_\sigma^+ X^+) \right|
  = \left|\rho(\Phi_u X^+ ) \right|
  \le \rho(|\Phi_u X^+|).
\end{split} \]
(2) A typical trick to break this uniformity assumption is changing $\Phi$ to $\Phi'=\Phi\circ f^n$.
If so, then the position of the peak value of the correlation is shifted by $n$;
that is, if $\cor_{\Phi,\diverg_\sigma^+  X^+}(0)$ is the largest correlation for $\Phi$,
then $\cor_{\Phi',\diverg_\sigma^+  X^+}(-n)$ is the largest correlation for $\Phi'$.
This breaks our assumption 2, 
which basically says that $\cor(0)$ is the largest.
However, this trick does not affect $\delta\mu(\Phi)$, 
and our bound in theorem~\ref{t:approximation} still works.
(3) 
Our work extends to slower, yet summable, decorrelation rates.
(4)
For numerical investigations of the decorrelation, see for example \cite{Casati1982}.
\end{remark*}

% holder
We choose to put $C_2$ and $\kappa_2$ as part of assumption~\ref{a2},
which can be numerically verified by the fast linear response algorithm for large $n$.
If $\Phi$, $X$, $f$, and the unstable subspace are known,
it is theoretically possible to derive the Holder norm of $\Phi$ and $\diverg_\sigma^+ X^+$,
and then derive an expression for $\kappa_2$.
If we further replace $L^1$ norms on the right side of the inequality by Holder norms,
it is then possible to write out $C_2$.
However, those expressions, though theoretically exist, 
would be too complicated, if not impossible, to compute.
Moreover, the bound by Holder norms can be very pessimistic:
for example, if $\Phi_u$ is orthogonal to $V^+$, or $X$ is parallel to $V^-$,
the unstable contribution would be zero, but the Holder bound can be large.
It is also possible to state a mixed bound, using $L^1$ norm for $n$ small and Holder norm for $n$ large.

\subsection{Estimating the unstable contribution} 
\label{s:error}
\hfill\vspace{0.1in}

% conditions for approximation; I2  is small when Ju is not aligned with Eu and fast decorrelate
In this subsection, we bound the unstable contribution, $\delta \mu(\Phi)$,
which is the systematic error of shadowing.
The rough ideas of our estimation have been described at the beginning of section~\ref{s:assump}.

We first define a norm.
For a measurable function, $g(\Phi,X, u)$, 
\[ \begin{split}
  \|g\| :=(\E(g^2))^{0.5} =(\E(\E(g^2|\Phi,X)))^{0.5},
\end{split} \]
where the expectation $\E$ is with respect to the joint distribution of $(\Phi,X, u)$,
with $u$ distributed according to the SRB measure $\rho$;
the conditional expectation $\E(\cdot|\Phi,X)=\rho(\cdot)$.
When a function does not depend on one of $(\Phi,X, u)$,
we still think of it as a three-variable function, and compute its norm.
For example, $\rho(g)$ does not depend on $u$, so
\[ \begin{split}
  \|\rho(g)\| 
  =(\E(\rho(g)^2))^{0.5}
  =(\E(\E(g|\Phi,X)^2))^{0.5}.
\end{split} \]
Notice that $\|\rho(g)\|\neq\|g\|$;
in fact, by Jensen's inequality, $(\rho(g))^2 \le \rho(g^2)$, hence 
\begin{equation} \begin{split} \label{e:jensen}
\|\rho(g)\|\le\|g\|.
\end{split} \end{equation}
For a vector field $v$, define norm $\|v\|:= \|\, |v|\, \|$, where $|\cdot|$ is the vector 2-norm.

\begin{lemma} \label{t:JuX+}
  Under assumption~\ref{a1},  
  \[ \begin{split}
    \frac {\| \Phi_uX^+ \|} {\|\Phi_uX\|} \le \frac 1{\sin\alpha} \sqrt{\frac mM},
  \end{split} \]
  where $\alpha$ is the smallest angle between stable and unstable subspace on the attractor.
\end{lemma}

\begin{remark*}
(1)
Here $\|\Phi_uX\|$ is an estimation of the magnitude of the true sensitivity.
(2)
This lemma can be generalized in several ways, for example,
$\alpha$ can be replaced by some kind of average instead of the lower bound,
assumption~\ref{a1} can also be replaced by more general models.
\end{remark*}

\begin{proof}
  By assumption, $X(U)$ and $\Phi_u$ have the same distribution for all $U$, hence
  \[ \begin{split}
  \E (\Phi_uX)^2 
  = \E (\sum_{j=1}^M \Phi_{u}^jX^j)^2
  = \E \E [(\sum_{j=1}^M \Phi_{u}^jX^j)^2| U] 
  = \E [(\sum_{j=1}^M \Phi_{u}^jX^j)^2| U] ,
  \end{split} \]
  where $X^j$ is the $j$-th coordinate of $X$.
  By independence, $\E [\Phi_{u}^iX^j\Phi_{u}^kX^l |U]=0$ unless $i=k$ and $j=l$.
  Hence,
  \begin{equation} \begin{split} \label{e:bet}
  \E (\Phi_uX)^2 
  = \sum_{j=1}^M \E [(\Phi_{u}^jX^j)^2 | U]
  = M
  \quad\Rightarrow\quad
  \|\Phi_uX\| = \sqrt M
  .
  \end{split} \end{equation}
  Denote the entries in the oblique projection matrix $P^+$ by $P^+_{ij}$, then
  \[ \begin{split}
  \E (\Phi_uX^+)^2 
  &= \E (\Phi_uP^+X)^2
  = \E (\sum_{i,j} \Phi_{u}^iP^+_{ij}X^j)^2
  = \E\E[ (\sum_{i,j} \Phi_{u}^iP^+_{ij}X^j)^2 |U]\\
  &= \E \sum_{i,j} \E [(\Phi_{u}^iP^+_{ij}X^j)^2 |U]
  = \rho \left(\sum_{i,j} (P^+_{ij})^2\right).
  \end{split} \]
  The orthogonal invariance of Frobenius norm says that, 
  for any $M\times M$ orthogonal matrix $A$,
  \[
  \sum_{i,j} (P^+_{ij})^2
  = tr(P^{+T}P^+)
  = tr((P^+A)^{T}(P^+A))
  = \sum_{i,j} (P^+A)_{ij}^2.
  \]
  Let the first $m$ and the \ree{remaining} $M-m$ columns of $A$
  be \ree{an} orthonormal basis of $(V^-)^\perp$ and $V^-$,
  where $(V^-)^\perp$ is the orthogonal complement of $V^-$.
  Then, only the first $m$ columns of $P^+A$ are non-zero,
  and their norms are bounded by $1/\sin \alpha$.
  Hence,
  \[ \begin{split}
  \E (\Phi_uX^+)^2 
  = \rho \left(\sum_{i,j} (P^+A)_{ij}^2\right)
  = \rho \left(\sum_{j} |(P^+A)_{j}|^2\right)
  = \rho \left(\sum_{1 \le j\le M} |(P^+A_j)|^2\right)
  \\
  = \rho \left(\sum_{1 \le j\le m} |(P^+A_j)|^2\right)
  \le \rho\left( \frac m{(\sin\alpha)^2} \right) 
  = \frac m{(\sin\alpha)^2}
  ,
  \end{split} \]
  where $(\cdot)_j$ is the $j$-th column vector,
  and $|\cdot|$ is the vector 2-norm.
  The lemma is proved by dividing by equation~\eqref{e:bet}.
\end{proof}

\begin{theorem} [error of shadowing] \label{t:approximation}
  Under assumption~\ref{a1} and~\ref{a2},
  \[
    \frac{\|\delta \mu(\Phi)\|} {\|\Phi_uX\|} 
    \le \left(\frac {C_1} {(1-\lambda)\sin\alpha}
      + \frac {C_2 \kappa_2} {(1-\kappa_2)\sin\alpha}
    \right) \sqrt\frac{m}M.
  \]
\end{theorem}

\begin{remark*}
  (1) 
  We may as well write $\delta^{(2)} \avg{\Phi}$ in the same form as $\delta^{(3)} \avg{\Phi}$,
  and bound it by decay of correlations.
  But if $\lambda<\kappa_2$ or $C_1<C_2$, our bound here can be sharper.
  (2)
  Our estimation here also bounds the error of the blended response algorithm.
  Blended response introduces approximations on the unstable contribution,
  so its error should be somewhat smaller than shadowing,
  although it is difficult to quantify the error more accurately without extra assumptions.
  (3) 
  To generalize this theorem, 
  we may replace the lower bound of $\kappa_2$ and $\lambda$ by some kind of average.
  Slow decorrelation or decay \ree{will} not only affect shadowing methods;
  they make most theories and computations related to SRB measures difficult.
  (4)
  For a given application,
  the posterior error of shadowing can be obtained by comparing with finite differences
  or the fast linear response algorithm.
\end{remark*}

\begin{proof}
  Set $N=1$ in equation~\eqref{e:ds123}. 
  First notice that the exponential decay of terms in $\delta^{(2)} \avg{\Phi}$ 
  is given by propagating unstable vectors  \ree{backward} in time.
  Note that $\Phi_u(f^n(u))$ and $X(u)$ are independent by assumption~\ref{a1}, 
  we have
  \[
   \left\| \ip{\grad (\Phi\circ f^{n} ),X^+} \right\|^2
   = \left\| \Phi_u  f_*^{n} P^+ X \right\|^2
   = \rho\left(\sum_{i,j} ( f_*^{n}P^+)_{ij}^2\right).
  \]
  Use the same $A$ as in the proof of lemma~\ref{t:JuX+}, 
  then use the fact that the non-zero columns in $P^+A$ are in the unstable subspace,
  and $f_*^n$ reduces their norms for $n\le0$,
  \[
   \rho\left(\sum_{i,j} ( f_*^{n}P^+)_{ij}^2\right)
   = \rho\left(\sum_{i,j} ( f_*^{n}P^+A)_{ij}^2\right)
   = \rho\left(\sum_{1\le j\le m} |f_*^{n}P^+A_{j}|^2\right)
   \le C_1^2 \lambda^{-2n} \frac{m} {(\sin\alpha)^2}.
  \]
  Hence, by equation~\eqref{e:jensen},
  \[
   \left\|\rho \ip{\grad (\Phi\circ f^{n} ),X^+} \right\|
   \le \left\| \ip{\grad (\Phi\circ f^{n} ),X^+} \right\|
   \le C_1 \lambda^{-n} \sqrt m/  \sin\alpha .
  \]
  On the other hand,
  the exponential decay of terms in $\delta^{(3)} \avg{\Phi}$ is due to the decorrelation,
  with the rate given by assumption~\ref{a2}.
  \[ \begin{split}
    \left\|\rho \left((\Phi\circ f^n ) \diverg_\sigma^+  X^+\right)\right\| 
    \le C_2 \kappa_2^n \| \rho(|\Phi_u X^+|) \|
    \le C_2 \kappa_2^n\|\Phi_u X^+\|.
  \end{split} \]
  Further use the estimation of $\|\Phi_u X^+\|$ in lemma~\ref{t:JuX+}, we have
  \[
    \left\|\rho \left((\Phi\circ f^n ) \diverg_\sigma^+  X^+\right)\right\| 
    \le C_2 \kappa_2^n\sqrt m/  \sin\alpha. 
  \]
  Finally, the error of shadowing methods is bounded by sums of two geometric series.
\[\begin{split}
  \left\|\delta^{(2)} \avg{\Phi}\right\|
  \le \sum_{n\le 0} \left\| \rho \ip{\grad (\Phi\circ f^{n} ),X^+} \right\|
  \le \frac {C_1\sqrt m} {(1-\lambda)\sin\alpha}  ;\\
  \left\|\delta^{(3)} \avg{\Phi}\right\|
  \le \sum_{n\ge 1}  \left\|\rho \left((\Phi\circ f^n ) \diverg_\sigma^+  X^+\right)\right\|
  \le \frac {C_2 \kappa_2\sqrt m} {(1-\kappa_2)\sin\alpha}.
\end{split}\]
The proof is completed by the fact that
$\delta \mu(\Phi) = \delta^{(2)} \avg{\Phi}+\delta^{(3)} \avg{\Phi}$.
\end{proof}

By our estimation, an interesting scenario where shadowing methods have small error 
is when the unstable ratio $m/M\ll 1$.
This is typically the case for systems with dissipation, such as fluid mechanics,
where nonintrusive shadowing is successful 
\cite{Ni_CLV_cylinder,Ni_nilsas,Ni_NILSS_JCP,Blonigan_2017_adjoint_NILSS,Chandramoorthy2018_nilss_ad}.
In fact, \ree{the} SRB measure was invented for dissipative systems,
many of which have low dimensional unstable subspaces.
However, there are examples with a large unstable ratio where shadowing methods fail,
such as the expanding circle to be discussed later.
A remedy to reduce the systematic error is given in the next subsection.

\subsection{Corrections to shadowing methods} 
\hfill\vspace{0.1in}
\label{s:correction}

% when shadowing method work and not work
When the systematic error of the shadowing method, the unstable contribution, is large, 
it can be reduced by further adding $\delta^{(2)} \avg{\Phi}$, defined in equation~\eqref{e:ds123},
to the result of shadowing methods.
This correction reduces \ree{but does not eliminate} the systematic error of shadowing.
By the proof of theorem~\ref{t:approximation}, the relative error is reduced to
\[
   \frac{\left\|\delta^{(3)} \avg{\Phi}\right\|} {\|\Phi_uX\|} 
  \le \frac {C_2 \kappa_2^N} {(1-\kappa_2)\sin\alpha} \sqrt \frac mM.
\]
In fact, earlier work on shadowing methods suggested that relaxing
the constraint in the optimization could improve the accuracy \cite{Blonigan_MSS}.
By our current analysis, we now know that is because relaxing the constraint 
may leak some unstable vectors into the shadowing direction computed by shadowing methods:
this is equivalent to adding some unstable contributions.

When $N$ is small, the trajectory previously used in the shadowing method is long enough
to average out the noise caused by the exponentially growing integrands in 
$\delta^{(2)}\Phi_{avg}$,
and the cost does not change with $N$.
Further increasing $N$ exhausts the unstable contribution;
however, when the previous trajectory is not long enough,
this incurs large computational cost.
The asymptotic cost of this correction
is significantly higher than the fast linear response algorithm,
whose integrand grows only as $O(\sqrt N)$ \cite{flr}.
However, the correction here does not require 
the heavy derivation as in the fast linear response algorithm,
hence it can be more easily generalized to, for example, continuous time and adjoint versions.

% example: sawtooth map
We illustrate the error of shadowing and the efficacy of the correction term 
on the 1-dimensional sawtooth map, or the expanding circle,
which was previously used as a counter example of shadowing methods \cite{BloniganPhdThesis}.
\re{It is also the underlying source of chaos for several other counter examples 
such as the solenoid map.}
Now we know that shadowing methods fail because the only dimension is unstable,
which means the unstable ratio, $m/M$, is as large as it can be.
Due to the very fast decay of correlation,
the proposed correction accurately computes the unstable contribution with a small $N$.

\begin{example*}[expanding circle]
Consider the dynamical system on $[0,2\pi)$ given by 
\[ \begin{split}
u_{k+1} = f(u_k,\gamma) := 2 u_k + \gamma \sin u_k \pmod{2\pi},
\quad
\Phi(u) := \cos u    .
\end{split} \]
The base parameter is $\gamma=0$, at which we compute the derivative.
Although this map is 2-to-1 rather than a diffeomorphism,
the linear response formula is still correct \cite{Baladi2007}.

The SRB measure $\rho$ of a 2-to-1 map is still defined as the long-time limit of evolving the Lebesgue measure.
However, $f^{n}(\cdot)$ is no longer a function for $n<0$,
for example, $f^{-1}x$ can be either $x/2$ or $x/2+\pi$.
For a random variable $U$ distributed according to $\rho$,
$(\{U_n:=f^{n}(U)\}_{n\le0} \,|\, U)$ is a reversed Markov chain,
with $U_{n-1}$ equally distributed given $U_n$.
More specifically, for $n\le0$, the conditioned probability
\[ \begin{split}
  \Pro\left(U_{n-1}=\frac12 U_n\mid U_n\right)=
  \Pro\left(U_{n-1}=\frac12 U_n+\pi \mid U_n\right) = \frac 12.
\end{split} \] 
 
Since there is no stable subspace, 
\[ \begin{split}
  X^+(U)=X(U)
  =\sin(U_{-1}).
\end{split} \]
By the chain rule, 
\[ \begin{split}
  \grad (\Phi\circ f^{n} )(U)
  =-2^n\sin(U_n).
\end{split} \] 
Hence,
\[ \begin{split}
  \ip{\grad (\Phi\circ f^{n} ),X^+}
  =-2^n\sin(U_n) \sin(U_{-1}).
\end{split} \]

For this example, shadowing with correction gives the true derivative for any $N\ge0$.
To show this, we only need to check that each term in $\delta^{(3)} \avg{\Phi}$ is zero.
For $n\ge 0$, $U_n=2^nU$ is a well-defined function of $U$, 
and the $n$-th term in $\delta^{(3)} \avg{\Phi}$ is
\[ \begin{split}
  &-\rho \ip{(\Phi\circ f^n ) \diverg_\sigma^+  X^+}
= \rho \ip{\grad (\Phi\circ f^{n} ),X^+}
\\
=& -\E (2^n\sin(2^nU)\sin U_{-1})
= -\E(2^n\sin(2^nU) \E(\sin U_{-1} \mid U))
=0.
\end{split} \]
In this example, $\E$ means to take expectation only with respect to $\rho$,
since $X$ and $\Phi$ are given.

For better understanding, we also directly compute
the linear response, which we now know equals $\delta^{(1)} \avg{\Phi}+ \delta^{(2)} \avg{\Phi}$,
since $\delta^{(3)} \avg{\Phi}=0$.
For $\delta^{(2)} \avg{\Phi}$,
when $n\le-2$, 
\[
  \rho \ip{\grad (\Phi\circ f^{n} ), X^+}
= -\E (2^n\sin U_n\sin U_{-1})
= -\E(2^n\sin U_{-1}\E (\sin U_n \mid U_{-1}))
=0 .
\] 
The only non-zero term is $n=-1$,
\[ \begin{split}
  \rho \ip{\grad (\Phi\circ f^{-1} ), X^+}
  =- \frac 12 \E ((\sin U_{-1})^2)
  =- \frac 12 \E\E ((\sin U_{-1})^2|U) \\
  =- \frac 12 \rho \left(\frac 12\sin^2\frac u2 + \frac 12\sin^2\frac {u+2\pi}2\right)
  =- \frac 12 \rho \left(\sin^2\frac u2\right)
  = -\frac 14.
\end{split} \]
Hence $\delta^{(2)} \avg{\Phi} = -1/4$.
By the same computations as above, using the definition in equation~\eqref{e:ds123},
we can see that the shadowing contribution $\delta^{(1)} \avg{\Phi} =  1/4$.
Hence the linear response is $\delta^{(1)} \avg{\Phi}+ \delta^{(2)} \avg{\Phi}=0$.

Finally, as a verification of our theorem~\ref{t:equal}, 
we directly compute $v$ and $\delta^{sd}\avg{\Phi}$,
\[ \begin{split}
  v(u) =  - \sum_{n\le -1} f_*^n X_{-n}^+ (u)
  =  - \sum_{n\le -1} 2^n \sin(2^{-n-1}u) 
  =  - \sum_{l\ge 0} 2^{-l-1} \sin(2^lu) .
\end{split} \]
Hence, by definition, the target of shadowing methods is
\[ \begin{split}
 \delta^{sd} \avg{\Phi} := \rho(\Phi_u v)
 = \rho\left(\sin u \sum_{l\ge 0} 2^{-l-1} \sin(2^lu)\right) 
 = \frac 12 \rho\left(\sin^2 u \right) 
 = \frac 14.
\end{split} \]
Here all terms with $l\ge1$ are zero, because
\[ \begin{split}
 \rho\left(\sin u \sin(2^lu)\right) 
 = \rho\left(\sin (u-\frac \pi 2) \sin(2^l(u-\frac\pi2))\right) 
 = \pm \rho\left(\cos u \sin(2^lu) \right)  \\
 = \pm \int_0^{2\pi}\cos u \sin(2^lu) \, du 
 = \pm \int_{-\pi}^{\pi}\cos u \sin(2^lu) \, du
 = 0,
\end{split} \]
where the positive sign is taken only when $l=1$,
and the last equality is because the integrand is an odd function.
Hence, $\delta^{sd} \avg{\Phi} =\delta^{(1)} \avg{\Phi} =  1/4$.
This is the same as the computational result in figure 2-17(a) of 
Blonigan's thesis \cite{BloniganPhdThesis},
where the interval was shrunk to $[0,1]$.
%%and $\Phi_u(u) = -2\pi\sin(u)$,
%so $\delta^{sd} \avg{\Phi}$ there should be $\pi/2$ by our analysis.
\qed
\end{example*}

% computing X+: prepare unstable tangent and adjoint CLVs
When $M>1$, $X^+$ can be efficiently computed by a `little-intrusive' algorithm,
which requires both tangent and adjoint solvers.
Denote the adjoint unstable subspace by $\aV^+$,
then $\dim\aV^+ = \dim V^+$, and $\aV^+\perp V^-$.
Moreover, both the unstable tangent and adjoint subspaces
can be obtained by evolving homogeneous tangent and adjoint equations \cite{Ni_adjoint_shadowing}.
To find $X^+$, just solve the vector such that
\[ \begin{split}
  X^+\in V^+, \quad \ip{X-X^+,\aV^+}=0.
\end{split} \]
With $\{w_i\}_{i=1}^m$ as the basis of $V^+$,
we can write $X^+$ as $X^+ =\sum_{i=1}^m c_iw_i$, 
then there are exactly $m$ linear equations for $m$ undetermined coefficients, $\{c_i\}_{i=1}^m$.
The cost of the little-intrusive algorithm is only $O(m)$;
in contrast, the blended response algorithm also requires computing $X^+$,
but was done with cost $O(M)$ \cite{abramov2007blended}.

\section{Convergence of nonintrusive shadowing}
\label{s:NIS}

In this section we prove the convergence of the nonintrusive shadowing algorithm,
given in equation~\eqref{e:nilss},
to the shadowing contribution.
The error of this convergence further includes two parts:
the first is the difference between the true shadowing direction 
and the one computed by nonintrusive shadowing;
the second is the sampling error caused by computing the shadowing contribution
from the true shadowing direction, but on a finite trajectory.
Together with the estimation of the unstable contribution in section~\ref{s:error},
we have the total error of approximating the linear response via nonintrusive shadowing algorithms.

\subsection{Auxiliary terms \texorpdfstring{$v', v^N, v^P, v^A, e^N, e^P, e^{PN}, \tilde e$}
{v's and e's}}
\hfill\vspace{0.1in}

% extra pre-process: make the W exactly the unstable subspace
In this section, we assume that in the nonintrusive shadowing algorithm 
in equation~\eqref{e:nilss},
\[ \begin{split}
  \spanof (w_1, \cdots, w_u)  = V^+.
\end{split} \]
This assumption can be achieved by evolving $w_i$'s for some time before the zeroth step,
since the unstable components in $w_i$'s grow faster than stable components.
In reality, such pre-process is typically not needed for nonintrusive shadowing to converge,
but making this assumption simplifies our theoretical analysis.
Should we want to extend our analysis to cases without this pre-process,
we need a sharp estimation of the unstable components 
in the random initial conditions of $w_i$'s.

% define terms
We start with some definitions.
Denote the total number of steps by $K$.
Let $v$ be the shadowing direction in equation~\eqref{e:shadowing_diffeo}.
In the nonintrusive shadowing algorithm,
let $v'$ be 
\[ \begin{split}
  v'_k := \sum_{0\le n \le k-1}f_*^nX_{k-n} \,.
\end{split} \]
We will show that $v'$ is the inhomogeneous tangent solution 
solved from the zero initial condition.
Moreover, let $v^P$ be the pivot solution defined by
\[ \begin{split}
  v^P_k 
  := \sum_{0\le n \le k-1}f_*^nX_{k-n}^- - \sum_{n\le -1} f_*^n X_{k-n}^+ \,.
\end{split} \]
We will show that $v^P$ is in the feasible set of nonintrusive shadowing, 
and also close to both $v$ and $v^N$,
where $v^N$ is the solution of the nonintrusive shadowing algorithm.
Define $v^A$, which bounds both $v$ and $v^P$, by
\begin{equation} \begin{split} \label{e:dangd}
  v^A_k 
  := \sum_{0\le n}|f_*^nX_{k-n}^-| \nax{+} \sum_{n\le -1} |f_*^n X_{k-n}^+ | \,,
\end{split} \end{equation}
where $|\cdot|$ is the vector 2-norm.
$v^A$ and $v$ are equivariant, that is,
\[ \begin{split} \label{e:zhujin}
  v^A_k = v^A_0\circ f^k.
\end{split} \]
However, notice that $v^P$ is not equivariant:
that is why we will mostly bound it by $v^A$.
Moreover, we define the errors
\[ \begin{split}
  e^N:= v^N - v \,,\quad
  e^P:= v^P - v \,,\quad
  e^{PN}:= v^P -v^N \,.
\end{split} \]
Here $e^N$ is the error of the shadowing direction computed by nonintrusive shadowing.
On a finite trajectory, the shadowing contributions computed by $v$ and $v^N$ are different by 
\[ \begin{split}
  \tilde e^N := \frac 1K \sum _{k=0} ^{K-1} \Phi_{uk} e^N_k .
\end{split} \]

We give some basic properties of the auxiliary terms we just defined.

\begin{lemma}  \label{l:homoinhomo}
  $v, v'$, and $v^P$ are inhomogeneous tangent solutions 
  satisfying equation~\eqref{e:inhomo_tangent_diffeo};
  $v'_0 =0$;
  $v^P$ is in the feasible set of nonintrusive shadowing, that is, $v^P- v'\in V^+$.
  $e^N, e^P$, and $e^{NP}$ are homogeneous tangent solutions
  satisfying equation~\eqref{e:homo_tangent_diffeo}.
\end{lemma}

\begin{proof}
  To see $v^P$ is inhomogeneous tangent, apply definitions,
  \[ \begin{split}
  v^P_{k+1} - f_*v^P_k
  &= \sum_{0\le n \le k}f_*^nX_{k+1-n}^- - \sum_{n\le -1} f_*^n X_{k+1-n}^+ 
  - \sum_{0\le n \le k-1}f_*^{n+1} X_{k-n}^- + \sum_{n\le -1} f_*^{n+1} X_{k-n}^+  \\
  &= \sum_{0\le n \le k}f_*^nX_{k+1-n}^- - \sum_{n\le -1} f_*^n X_{k+1-n}^+ 
  - \sum_{1\le l \le k}f_*^{l} X_{k+1-l}^- + \sum_{l\le 0} f_*^{l} X_{k+1-l}^+\\
  &= X_{k+1}^- +  X_{k+1}^+  = X_{k+1}.
  \end{split} \]
  Similarly we can verify that $v$, defined by equation~\eqref{e:shadowing_diffeo},
  and $v'$, are inhomogeneous tangent solutions.
  Also, by definitions, $v'_0 =0$, and 
  \[ \begin{split}
    v^P_k -v'_k =  - \sum_{n\le k-1} f_*^n X_{k-n}^+ \in V^+_k.
  \end{split} \]
  Finally, 
  $e^N, e^P$, and $e^{NP}$ are homogeneous tangent solutions,
  since they are differences between inhomogeneous tangent solutions.
\end{proof}

\subsection{Convergence of \texorpdfstring{$v^N$ to $v$}{vN to v}}
\hfill\vspace{0.1in}

In this subsection, we show that the result of nonintrusive shadowing, $v^N$,
converges to the true shadowing direction, $v$, as the trajectory length $K\rightarrow\infty$.
We also show the difference between the shadowing contributions computed by $v$ and $v^N$
converges to zero.
More specifically, 
we will bound  $e^{PN}\in V^+$ at the last step of the trajectory,
and bound  $e^P\in V^-$ at the first step.
Then, due to the exponential decay of unstable and stable vectors,
$e^N = e^{PN} + e^{P}$ converges to zero at the middle part of the trajectory,
and the averaged error, $\tilde e^N$, also converges to zero as $K\rightarrow \infty$.
We shall first prove the convergence for chosen $\Phi$ and $X$,
then we give a quantitative bound on $\|\tilde e^N\|$ 
for $\Phi$ and $X$ distributed according to assumption~\ref{a1}.

% lemma, not involve distribution of  \Phi and X
\begin{lemma}  \label{l:ddddd}
$e^{PN}\in V^+$, $e^{P}\in V^-$, and their peak values are bounded by
  \[ \begin{split}
  |e^{PN}_{K-1}|
  \le C_1 \sum_{k=0}^{K-1} \lambda^{K-1-k} v^A_k 
  \,,\quad
  |e_0^P| \le v_0^A.
  \end{split} \]
\end{lemma}

\begin{remark*}
The main idea for bounding $e^{PN}\in V^+$ is that 
the unstable homogeneous tangent solution has a spike at $(K-1)$-th step,
hence $e^{PN}$ can not to be too large without increasing the $l^2$ norm,
$|v^N|_K:=(\sum_{k=0}^{K-1} |v_k|^2 )^{0.5}$.
Hence minimizing $|v^N|_{K}$ controls $e^{PN}$.
Here the `large spike' is encoded into the relation $|e^{PN}|_{K}\approx |e^{PN}_{K-1}|$.
\end{remark*}

\begin{proof}
By definitions $v^{P}-v'\in V^+$ and  $v^N-v'\in V^+$, hence 
\[ \begin{split}
  e^{PN} := v^P-v^N\in V^+.
\end{split} \]

Since $|v^N|_K$ is minimized in nonintrusive shadowing,
for any $w\in V^+$, $|v^N+\alpha w|_K^2$ is minimal at $\alpha=0$.
Differentiate with respect to $\alpha$,
we have the so-called first-order optimality condition,
\begin{equation} \begin{split} \label{e:so}
  \ip{v^N, w}_K :=\sum_{k=0}^{K-1} \ip{v^N_k, w_k}   =0 ,\quad \textnormal{ for all } w\in V^+.
\end{split} \end{equation}
Substitute $w=e^{PN}$ and $v^N=v^P-e^{PN}$ into equation~\eqref{e:so}, we have
\[ \begin{split}
  \ip{v^P - e^{PN}, e^{PN}}_K  =0 
  \quad\Rightarrow \quad
  \ip{ e^{PN}, e^{PN}}_K = \ip{ e^{PN}, v^P}_K .
\end{split} \]
The peak value of $e^{PN}$ is at step $K-1$, 
which is smaller than its $l^2$ norm, hence
\[ \begin{split}
  |e^{PN}_{K-1}|^2
  \le \ip{ e^{PN}, e^{PN}}_{K} 
  = \ip{ e^{PN}, v^P}_K.
\end{split} \]
Apply Cautchy-Schwarz and the exponential decay of $e^{PN}$, we have
\[ \begin{split}
  |e^{PN}_{K-1}|^2
  \le \ip{ e^{PN}, v^P}_K
  \le \sum_{k=0}^{K-1} |e^{PN}_k| |v^P_k|
  \le  C_1 \sum_{k=0}^{K-1} \lambda^{K-1-k} |e^{PN}_{K-1}| |v^P_k|.
\end{split} \]
Cancel $|e^{PN}_{K-1}|$ from both sides, we get
\[ \begin{split}
  |e^{PN}_{K-1}|
  \le C_1 \sum_{k=0}^{K-1} \lambda^{K-1-k} |v^P_k| 
  \le C_1 \sum_{k=0}^{K-1} \lambda^{K-1-k} v^A_k.
\end{split} \]

To prove the statements about $e^P$,
notice that by definition, 
\[ \begin{split}
  e^P_k 
  = \sum_{n \ge k}f_*^nX_{k-n}^- \in V^-_k
  \,.
\end{split} \]
The second inequality in the lemma is due to the definition of $v^A$.
\end{proof}

\begin{lemma} [convergence of $v^N$ to $v$]
For chosen $X$ and $\Phi$,
\[ \begin{split}
  |e^N_k| \le  \frac{4C_1^3 |X|_{max} \lambda^{\min(K-1-k,k)}}{(1-\lambda)^2 \sin\alpha}
  ,\quad
  |\tilde e^N | 
  \le \frac {|\Phi_u|_{max}} K \sum _{k=0}^{K-1}  |e^N_k|
  \le \frac {4 C_1^3 |X|_{max}|\Phi_u|_{max} } {K(1-\lambda)^3 \sin \alpha} ,
\end{split} \]
where $|\cdot|_{max}$ is the largest vector 2-norm on the attractor.
\end{lemma}

\begin{remark*}
(1)
The first inequality means the convergence of $v^N$ happens at the middle part
of the trajectory; its error at either end of the trajectory does not shrink with $K$.
The good news is that the $l^1$ norm of $e^N$ also does not increase with $K$.
This bound on $e^N$ is useful when nonintrusive shadowing
is used for computing only the shadowing direction but not the shadowing contribution,
for example, when computing the modified shadowing direction in the fast linear response algorithm
\cite{flr}.
(2)
Previous shadowing methods have the same $O(K^{-1})$ convergence speed for $\tilde e^N$
\cite{wang2014convergence,Blonigan_MSS};
hence, the nonintrusive formulation reduces the computation with no additional error.
Also, the convergence to the linear response in previous shadowing methods' literature was wrong,
it should be convergence to the shadowing contribution.
\end{remark*}

\begin{proof}
We first bound $v^A$. Use its definition in equation~\eqref{e:dangd},
\begin{equation} \begin{split} \label{e:dang}
  |v^A|
  \le \sum_{0\le n} |f_*^nX_{-n}^-| + \sum_{n\le -1} |f_*^n X_{-n}^+ | 
  \le C_1 \sum_{0\le n} \lambda^n |X_{-n}^-| + C_1 \sum_{n\le -1} \lambda^{-n} | X_{-n}^+ |.
\end{split} \end{equation}
Bound $|X|$ by its maximum, we have
\[ \begin{split}
  &|v^A|
  \le \frac {2C_1|X|_{max}} {(1-\lambda)\sin \alpha}.
\end{split} \]

Since $e^N=e^{PN}+e^P$, where $e^{PN}\in V^+$, $e^P\in V^-$,
and by lemma~\ref{l:ddddd}, we have
\[ \begin{split}
  |e^N_k|
  &\le |e^{PN}_k| + |e^{P}_k|
  \le C_1 (\lambda^{K-1-k} |e^{PN}_{K-1}| + \lambda^k |e^{P}_0| ) 
  \le C_1 \lambda^{\min(K-1-k,k)} (|e^{PN}_{K-1}| + |e^{P}_0| ) \\
  &\le  \frac{2C_1^2 \lambda^{\min(K-1-k,k)}}{(1-\lambda)} |v^{A}|_{max} 
  \le  \frac{4C_1^3 |X|_{max} \lambda^{\min(K-1-k,k)}}{(1-\lambda)^2 \sin\alpha}.
\end{split} \]

For $\tilde e^N$, 
\[ \begin{split}
  |\tilde e^N | \le \frac {|\Phi_u|_{max}} K \sum _{k=0}^{K-1}  |e^N_k|.
\end{split} \]
We have a slightly cleaner bound for the average of $e^N_k$,
\begin{equation} \begin{split} \label{e:guzheng}
  \frac 1 K \sum _{k=0}^{K-1}  |e^N_k|
  \le \frac 1 K \sum _{k=0}^{K-1}  |e^{PN}_k| + |e^{P}_k|
  \le \frac {C_1}K \sum _{k=0}^{K-1} \lambda^{K-1-k} |e^{PN}_{K-1}| + \lambda^k |e^{P}_0| \\
  \le \frac {C_1}{K(1-\lambda)} ( |e^{PN}_{K-1}| + |e^{P}_0| )
  \le \frac {C_1}{K(1-\lambda)} \left(C_1 \sum_{k=0}^{K-1} \lambda^{K-1-k} | v^A_k |
  + |v_0^A| \right)
\end{split} \end{equation}
The lemma is proved by substituting the bound for $v^A$.
\end{proof}

% theorem, with assumption on distribution of \Phi and X
\begin{theorem} [convergence of $v^N$ to $v$]
\label{t:convergeNI}
For the distribution of $\Phi$ and $X$ given in assumption~\ref{a1}, 
\[ \begin{split}
  \frac {\| \tilde e^N \|} {\| \Phi_uX \|}
  \le \frac 1{K\|X\|} \sum _{k=0}^{K-1} \|e^N_k\|
  \le \frac{4 C_1^3}{K(1-\lambda)^3 \sin \alpha} .
\end{split} \]
\end{theorem}

\begin{proof}
By definition,
\[ \begin{split}
  \| \tilde e^N \| 
  \le \frac 1K \sum _{k=0}^{K-1} \| \Phi_{uk} e_k^N \|
  = \frac 1K \sum _{k=0}^{K-1} 
  \left [\E \E \left((\Phi_{uk} e_k^N) ^2 \mdd
  \{u_k\}_{ k=0 }^K, \{ X_k \}_{ k=0 }^K\right)\right]^{0.5}.
\end{split} \]
Here $e^N$ is determined by $\{ X_k \}_{ k=0 }^K$ and $\{u_k\}_{ k=0 }^K$.
\nax{
This is the place where we use the full strength of the independence condition
in assumption~\ref{a1};
that is, conditioned on the entire sequence of $u$ and $X$, 
$\Phi_{uk}$ is still a multi-variate Gaussian.
}

We choose a coordinate whose first axis is parallel to $e_k^N$,
then $\Phi_{uk}$ is still multi-variate Gaussian in this new coordinate.
In particular, its first coordinate, $\Phi_{uk}^1\sim \N(0, 1)$,
whereas other coordinate components are orthogonal to $e_k^N$.
Hence,
\begin{equation} \begin{split} \label{e:ohoh}
  \E \left((\Phi_{uk} e_k^N) ^2  \mdd \{u_k\}_{ k=0 }^K, \{ X_k \}_{ k=0 }^K\right)
  = \E \left(|e_k^N|^2 (\Phi^1_{uk})^2 \mdd \cdots \right)
  = |e_k^N|^2 \E (\Phi^1_{uk})^2
  = |e_k^N|^2. 
\end{split} \end{equation}
By substitution,
\[ \begin{split}
  \| \tilde e^N \| 
  \le \frac 1K \sum _{k=0}^{K-1} \left(\E |e_k^N|^2 \right)^{0.5}
  = \frac 1K \sum _{k=0}^{K-1}  \|e^N_k\| .
\end{split} \]

Replace $|\cdot|$ by $\|\cdot\|$ in equation~\eqref{e:guzheng}.
Then, notice that $v^A$ is equivariant, 
so $\rho(v^A_k) = \rho(v^A_0)$, hence $\| v^A_k\| = \|v^A_0\|$, and
\[ \begin{split}
  &\| \tilde e^N \| \le
  \frac 1K \sum _{k=0}^{K-1}  \|e^N_k\|
  \le \frac {C_1}{K(1-\lambda)} ( \|e^{PN}_{K-1}\| + \|e^{P}_0\| )
  \le \frac {2C_1^2}{K(1-\lambda)^2} \|v_0^A\|.
\end{split} \]

To bound $\|v^A\|$,
notice that $X^-_{-n}(\cdot) := X^-\circ f^{-n}(\cdot)$, 
that is, $X_n^-$ is equivariant,
hence $\|X_{-n}^-\| = \|X^-\|$;
similarly, $\|X_{-n}^+\| = \|X^+\|$.
Replace $|\cdot|$ by $\|\cdot\|$ in equation~\eqref{e:dang},
we have
\begin{equation} \begin{split} \label{e:bababa}
  \|v^A_0\| 
  \le \frac {C_1( \|X^-\|+\|X^+\| )} {1-\lambda}
  \le \frac {2C_1\|X\|} {(1-\lambda)\sin \alpha}
\end{split} \end{equation}
Under assumption~\ref{a1}, $\|X\|=\sqrt M$, hence
\[ \begin{split}
  \| \tilde e^N \| 
  \le \frac 1K \sum _{k=0}^{K-1}  \|e^N_k\|
  \le \frac {2 C_1^2 \|v^A_0\| }{K(1-\lambda)^2}
  \le \frac {4 C_1^3 \|X\|} {K(1-\lambda)^3\sin \alpha}
  = \frac {4 C_1^3 \sqrt M} {K(1-\lambda)^3\sin \alpha}
\end{split} \]
By \nax{ equation~\eqref{e:bet} in} lemma~\ref{t:JuX+},
$\|\Phi_uX\|=\|X\| =\sqrt M$, hence this theorem is proved.
\end{proof}

\subsection{Sampling error on finite trajectories}
\label{s:sample}
\hfill\vspace{0.1in}

Even with the true shadowing direction, $v$,
computing the shadowing contribution on a finite trajectory subjects to sampling error.
This is the other error in computing the shadowing contribution.
After bounding this error in this subsection, 
we can finally bound the total error of computing the linear response by the nonintrusive shadowing algorithm on a finite trajectory.

For chosen $\Phi$ and $X$, the sampling error for taking average on a finite trajectory is 
\begin{equation} \begin{split} \label{e:es}
  \tilde e^S := \rho(\Phi_u v) - \frac 1K \sum _{k=0} ^{K-1} \Phi_{uk} v_k.
\end{split} \end{equation}
By proposition 1.2(c) in \cite{Ruelle_diff_maps},
$v=\delta j$ is Holder continuous.
Hence, by decay of correlations in uniform hyperbolic systems \cite{Chazottes2015},
we know that for any fixed $\Phi$ and $X$, $(\rho(\tilde e^S)^2)^{0.5}=O(K^{-0.5})$.
With assumption~\ref{a1}, and further with a statistical assumption on the decay of correlation of $\Phi_u v$,
we can give a quantitative bound on $\|\tilde e^S\|$.

\begin{assumption}\label{a3}
For the entire distribution of $\Phi$ and $X$,
there are uniform constants $C_3\ge 1, 0<\kappa_3<1$, such that
\[ \begin{split}
  \cor_{\Phi_u v}(n) := 
  \left|\rho((\Phi_u v\circ f^n ) \Phi_u v) -(\rho(\Phi_uv))^2 \right|
  \le C_3 \kappa_3^n \rho((\Phi_u v)^2).
\end{split} \]
\end{assumption}

\begin{remark*}
  Here $\rho((\Phi_u v)^2)$ is a bound for $\cor_{\Phi_u v}(0)$.
\end{remark*}

\begin{theorem} [convergence of sampling error] \label{t:sample error}
Under assumption~\ref{a1} and~\ref{a3},
\[ \begin{split}
  \frac{\|\tilde e^S\|}{\|\Phi_u X\|}
  \le \sqrt \frac {2C_3} {K(1-\kappa_3)} \frac {2C_1} {(1-\lambda)\sin \alpha}
  .
\end{split} \]
\end{theorem}

\begin{proof}
By some simple algebra, and that $\rho(\Phi_{uk}v_k) = \rho(\Phi_uv)$, we have
\[ \begin{split}
  \|\tilde e^S\|^2
  = \E \rho \left( \Big(\rho(\Phi_u v) - \frac 1K \sum _{k=0} ^{K-1} \Phi_{uk} v_k\Big)^2 \right)
  = \frac 1K \E \rho \left(\frac 1K \Big( 
    \sum _{k=0} ^{K-1} \rho(\Phi_u v) - \Phi_{uk} v_k\Big)^2 \right) \\
  = \frac 1K \E \left( 
    \cor_{\Phi_u v}(0) + 2\sum _{l=1} ^{K-1} \frac {K-l} K \cor_{\Phi_u v}(l) \right) 
  \le \frac 2K \E \left( \sum _{l=0} ^{K-1} \cor_{\Phi_u v}(l) \right) \\
  \le \frac {2C_3} {K(1-\kappa_3)} \E \rho \left( (\Phi_u v)^2 \right) 
  = \frac {2C_3} {K(1-\kappa_3)} \E \left( (\Phi_u v)^2 \right) ,
\end{split} \]
where all but the last $\E$ are averaging with respect to the distribution of $X$ and $\Phi$;
the last $\E$ further averages over $\rho$.
Then, by the same arguments as in equation~\eqref{e:ohoh},
because $\{u_k\}_{ k\in\Z }$ and $\{ X_k \}_{ k\in\Z }$ completely determine $v$, we have
\[ \begin{split}
  \E \left( (\Phi_u v)^2 \right) 
  = \E \E \left( (\Phi_u v)^2 \mdd \{u_k\}_{ k\in\Z }, \{ X_k \}_{ k\in\Z } \right) 
  = \E \left( |v|^2 \right) 
  \le \E \left( (v^A)^2 \right) .
\end{split} \]
Hence, 
\[ \begin{split}
  \|\tilde e^S\|
  \le \sqrt \frac {2C_3} {K(1-\kappa_3)} \|v^A\|
\end{split} \]
Then apply equation~\eqref{e:bababa} and~\eqref{e:bet} to prove the theorem.
\end{proof}

Now all errors in the nonintrusive shadowing algorithm have been analyzed.

\begin{theorem} 
The error of approximating the linear response by nonintrusive shadowing,
on a finite trajectory, under assumption~\ref{a1}, \ref{a2}, and~\ref{a3}, 
is bounded by 
\[ \begin{split}
  \| \delta \Phi_{avg} - \frac 1K \sum _{k=0} ^{K-1} \Phi_{uk} v^N_k\|
  \le \| \delta \mu(\Phi)\| + \|\tilde e^S\| + \|\tilde e^N\|. 
\end{split} \]
where the bounds of the three terms are given in 
theorem~\ref{t:approximation},~\ref{t:sample error}, and~\ref{t:convergeNI}, respectively.
\end{theorem}

\begin{remark*}
  Note that $\|\tilde e^S\|$ and $\|\tilde e^N\|$ go to zero as $K\rightarrow \infty$,
  but $\| \delta \mu(\Phi)\|$ does not.
\end{remark*} 

\begin{proof}
By triangular inequality,
\[ \begin{split}
  &\| \delta \Phi_{avg} - \frac 1K \sum _{k=0} ^{K-1} \Phi_{uk} v^N_k\| \\
  \le& \| \delta \Phi_{avg} - \rho(\Phi_u v) \| 
    + \| \rho(\Phi_u v) - \frac 1K \sum _{k=0} ^{K-1} \Phi_{uk} v_k\|
    + \| \frac 1K \sum _{k=0} ^{K-1} \Phi_{uk} v_k- \frac 1K \sum _{k=0} ^{K-1} \Phi_{uk} v^N_k\| \\
  =& \| \delta \mu(\Phi)\| + \|\tilde e^S\| + \|\tilde e^N\|. 
\end{split} \]
\end{proof}

\section{Conclusions}

For engineering applications, 
when computing derivatives of averaged objectives with respect to system parameters,
especially for dissipative systems with large degrees of freedom, such as computational fluids,
we suggest to first try the nonintrusive shadowing algorithm.
For many previous applications, nonintrusive shadowing can be quite useful even without correction.
When the unstable contribution is large, or when better accuracy is desired, for example near design optimal \cite{RepolhoCagliari2021,Lasagna2019,BloniganPhdThesis},
there are several choices for further computing the systematic error of shadowing, which is the unstable contribution.  
In particular, we can add adjoint solvers, then use the little-intrusive correction given in this paper;
another choice is to add a second-order inhomogeneous term to tangent equations, 
then use the fast linear response algorithm in \cite{flr}.

\section*{Acknowledgements}

The author gratefully thanks Miaohua Jiang and David Ruelle for discussions on the linear response, and also Adam Sliwiak for very helpful discussions.
This research is supported by the China Postdoctoral Science Foundation 2021TQ0016,
the International Postdoctoral Exchange Fellowship Program YJ20210018,
and the Richman Fellowship from the math department of UC Berkeley.

\bibliographystyle{abbrv}
{\footnotesize\bibliography{MyCollection}}

\end{document}